\documentclass[psamsfonts]{amsart}

\usepackage{amssymb,amsfonts,amsmath}
\usepackage[all,arc]{xy}
\usepackage{enumerate}
\usepackage{mathrsfs}
\usepackage{hyperref}
\usepackage{csquotes}

\newtheorem{theorem}{Theorem}[section]
\newtheorem{corollary}[theorem]{Corollary}
\newtheorem{proposition}[theorem]{Proposition}
\newtheorem{lemma}[theorem]{Lemma}

\theoremstyle{definition}
\newtheorem{definition}[theorem]{Definition}

\newtheorem{example}[theorem]{Example}

\newtheorem{remark}[theorem]{Remark}


\numberwithin{equation}{section}

\begin{document}
	\title[On Hardy's Apology Numbers]{On Hardy's Apology Numbers}
	
	\author[Henk Koppelaar]{\bfseries Henk Koppelaar}
	
	\address{Henk Koppelaar\\
		Faculty of Electrical Engineering, Mathematics and Computer Science\\
		Delft University of Technology\\
		Delft\\
		The Netherlands}
	\email{Koppelaar.Henk@gmail.com}
	
	\author[Peyman Nasehpour]{\bfseries Peyman Nasehpour}
	
	\address{Peyman Nasehpour\\
		Department of Engineering Science \\
		Golpayegan University of Technology\\
		Golpayegan\\
		Iran}
	\email{nasehpour@gut.ac.ir, nasehpour@gmail.com}
	
	
	
	\maketitle
	
	\begin{abstract}
		Twelve well known `Recreational' numbers are generalized and classified in three generalized types Hardy, Dudeney, and Wells. A novel proof method to limit the search for the numbers is exemplified for each of the types. Combinatorial operators are defined to ease programming the search.
	\end{abstract}
	
	\section{Introduction} 
	
	``Recreational Mathematics'' is a broad term that covers many different areas including games, puzzles, magic, art, and more \cite{Sigmaa2018}. Some may have the impression that topics discussed in recreational mathematics in general and recreational number theory, in particular, are only for entertainment and may not have an application in mathematics, engineering, or science. As for the mathematics, even the simplest operation in this paper, i.e. the sum of digits function, has application outside number theory in the domain of combinatorics \cite{ClementsLindstrom1965, Lindstrom1964,Lindstrom1965,Lindstrom1966,PatelSiksek2017} and in a seemingly unrelated mathematical knowledge domain: topology \cite{GitlerMahowaldMilgram1968,Hirsch1967,DavisMahowald1975}. Papers about generalizations of the sum of digits function are discussed by Stolarsky \cite{Stolarsky1977}. It also is a surprise to see that another topic of this paper, i.e. Armstrong numbers, has applications in ``data security'' \cite{DeepaKannimuthuKeerthika2011}. 
	
	In number theory, functions are usually non-continuous. This inhibits solving equations, for instance, by application of the contraction mapping principle because the latter is normally for continuous functions. Based on this argument, questions about solving number-theoretic equations ramify to the following:
	
	\begin{enumerate}
		
		\item Are there any solutions to an equation? 
		\item If there are any solutions to an equation, then are finitely many solutions?
		\item Can all solutions be found in theory?
		\item Can one in practice compute a full list of solutions?
		
	\end{enumerate}
	
	The main purpose of this paper is to investigate these constructive (or algorithmic) problems by the fixed points of some special functions of the form $f: \mathbb N \rightarrow \mathbb N$. We confine computations to the digits of a natural number $n$ in base $b$ (more interestingly in base $b=10$). Answers to the questions above are exemplified by proofs of theorems on some questions and equations in recreational number theory. The term recreation is kept here because of our bias to consider digits in the domain of $f$.
	
	The main trait of recreational questions in number theory is the puzzle aspect: if only a few rare solutions to an algorithmically easy question exist, then it often is recreational. Number theory, however, poses problems that cover universal properties, for instance, about sums and products in finite fields in general \cite{BourgainKatzTao2006}, or Blomer \cite{Blomer2008} asks if there exists a subset of $A$ of $\mathbb N$ such that every $n \not\equiv 0,4,7 \pmod 8$ can be represented as the sum of three squares in $A$? This has been done by Z\"{o}llner \cite{Zollner1985} for four squares, we have not found a reference to a full solution of this three squares question. Another illustrative example of this difference compared to problems in this paper is to find solutions of $n = (p_1 y_1)^2 + (p_2 y_2)^2 + (p_3 y_3)^2$ with $p_1, p_2, p_3$ odd and different primes, as large as possible such that the main term for the number of representations still dominates the error term \cite{Blomer2008}. This professional reasoning about error terms is not done in recreational texts.
	
	We narrow the gap between number theory and recreational number theory in this paper by adding a tool {\em to find the number of solutions to posed problems} for the recreationists. In view of the above mentioned facts, we give a brief sketch of the contents of this paper: In Section \ref{sec:Hardy}, based on some entertaining examples in Hardy's book \cite{Hardy1940} and other resources \cite{Pickover1995,vanBerkel2009}, we introduce base $b$ $F$-Hardy's apology numbers as follows: Let $F: \mathbb N_0 \rightarrow \mathbb N$ be a function, $b \in \mathbb N -\{1\}$ and $a_i$ be a non-negative integer number such that $a_i \leq b-1$ for all $0 \leq i \leq m-1$. We say that a natural number $n = \overline{a_{m-1} \cdots a_1 a_0}$ is a base $b$ $F$-Hardy's apology number if the following equality is satisfied: \[ H_1\colon n = \sum_{i=0}^{m-1} a_i b^i = \sum _{i=0}^{m-1} F(a_i). \qquad \text{($F$-Hardy's apology equation)}\]
	
	Then we show that for a specific $F$, the cardinality of all $F$-Hardy's apology numbers obtained from the $H_1$ equation mentioned above is finite (see Definition \ref{Fsum} and Theorem \ref{FsumPA}).
	
	Let us recall that an $n-$digit number in base $b$ is called a base-$b$ Armstrong number of $n^{\text{th}}$ order if it is equal to the sum of the $n^{\text{th}}$ power of its digits in base $b$ (see Definition \ref{Armstrongnumbers}). In Section \ref{sec:Armstrong}, we discuss Armstrong numbers and their generalization. It happens that Armstrong numbers have applications in ``data security'' \cite{Belose2012,Bansode2014,DeepaKannimuthuKeerthika2011}.
	
	If we denote the number of digits of $n$ in base $b$ by $D_b(n)$, then in Section \ref{sec:Wells}, we call a number to be a base $b$ $F$-Wells number, if $n = D_b(F(n))(=m)$, i.e., $n$ is the fixed point element of the function $D_b \circ F$, where $F: \mathbb N \rightarrow \mathbb N$ is a function and $F(n) = \sum_{i=0}^{m-1} a_i b^i$ the representation of $F(n)$ in base $b\geq 2$ (see Definition \ref{Wellsnumbers}). In Corollary \ref{FWellsnumbers}, we show that if $F: \mathbb N \rightarrow \mathbb N$ is a function, there are finitely many base $b$ $F$-Wells numbers, if one of the following statements holds:
	
	\begin{enumerate}
		
		\item There is a natural number $N$ such that $n \geq N$ implies that $F(n) \geq b^n$.
		
		\item There is a natural number $N$ such that $n \geq N$ implies that $F(n) < b^{n-1}$.
		
	\end{enumerate}
	
	Let us recall that if a non-negative integer $n$ has a representation $n = \sum_{i=0}^{m-1} a_i b^i$ in base $b \geq 2$, then the sum of the digits of $n$ is denoted by $S_b(n) = \sum_{i=0}^{m-1} a_i$ \cite[Theorem 6.5.1]{AndreescuAndrica2009}. Section \ref{sec:Dudeney} is devoted to a generalization of Dudeney numbers. Let $F: \mathbb N \rightarrow \mathbb N$ be a function and $F(n) = \sum_{i=0}^{m-1} a_i b^i$ be the representation of $F(n)$ in base $b\geq 2$. In Definition \ref{Dudeneynumbers}, we call a number $n$ to be a base $b$ $F$-Dudeney number, if $n = S_b(F(n))$, where $S_b(F(n)) = \sum_{i=0}^{m-1} a_i$. In Corollary \ref{DudeneyCor}, we prove that if $F: \mathbb N \rightarrow \mathbb N$ is a function such that $\displaystyle \lim_{n \rightarrow +\infty} \frac{F(n)}{b^{\frac{n-b+1}{b-1}}} = 0$, then there are finitely many base $b$ $F$-Dudeney numbers.
	
	In Section \ref{sec:PowersSumDigits}, we show that the number of natural numbers which are equal to the sum of their digits raised to a specific power is finite. In fact, in Theorem \ref{PowersSumDigits}, we prove that if $n$ is a natural number, $S_b(n)$ is the sum of the digits of the number $n$ in base $b$, $p \geq 2$ is a natural number, and $\phi_p : \mathbb N \rightarrow \mathbb N$ is defined with $\phi_p (n) = (S_b(n))^p$, then the number of natural numbers satisfying the equality $\phi_p (n) = n$ is finite. Moreover, we show that if $\phi_p (n) = n$ then $n \leq b^{p^2}$.
	
	This paper is in the continuation of our interest in Algorithms and Computation \cite{AghaieabianeKoppelaarNasehpour20171},\cite{AghaieabianeKoppelaarNasehpour20172},\cite{JelodarMoazzamiNasehpour2016},\cite{Nasehpour2020}.
	
	\section{A Generalization of Hardy's Apology Numbers}\label{sec:Hardy}
	
	English mathematician, Godfrey Harold Hardy (1877--1947), in his historical book on mathematics, with the title ``\emph{A Mathematician's Apology}'', said that ``there are just four numbers [after 1] which are the sums of the cubes of their digits''. Moreover, he said that ``8712 and 9801 are the only four-figure numbers which are integral multiples of their reversals'', and he also explained that these were not serious theorems, as they were not capable of any serious generalization  \cite{Hardy1940}. He did not imagine that the same numbers have become used in encryption \cite{Belose2012,Bansode2014,DeepaKannimuthuKeerthika2011}. Clifford A. Pickover, in his book \cite{Pickover1995} on page 169, defines a number to be factorion, if it is equal to the sum of the factorial values for each of its digits. Daan van Berkel on page 2 of his paper \cite{vanBerkel2009}, defines a number to be M\"{u}nchhausen in base $b$ if $n =  \overline{a_{m-1} \cdots a_1 a_0} = \sum_{i=0}^{m-1} a_i b^i $, then $n = \sum _{i=0}^{m-1} a_{i}^{a_i}$. An interesting  M\"{u}nchhausen number in base 10 is $3435 = 3^3 + 4^4 + 3^3 + 5^5$, the number that appears in the title of his paper. On the other hand, we can consider those numbers which are the sums of the cubes of their 2-grouped digits like the number $165033$ which is equal to $16^3 +50^3 + 33^3$ or $221859$ which is equal to $22^3+18^3+59^3$.
	The title given to his book reflects Hardy's modesty to colleagues and the world, as with the Persian ``Tarof'' (a Persian traditional modesty to refuse a gift, or benefit even if it is highly wanted). The tarof in mathematics is explained by Azarang \cite{Azarang2011}. We believe that Hardy is the first mathematician to express it and we honor him for it in the title of this paper.
	
	All these entertaining observations and amusing examples may inspire a curious mind to search for other numbers with similar properties. To this end, we give the following definition:
	
	\begin{definition}
		
		\label{Fsum}
		
		Let $F: \mathbb N_0 \rightarrow \mathbb N$ be a function, $b \in \mathbb N -\{1\}$ and $a_i$ be a non-negative 1913integer number such that $a_i \leq b-1$ for all $0 \leq i \leq m-1$. We say that a natural number $n = \overline{a_{m-1} \cdots a_1 a_0}$ is a base $b$ $F$-Hardy's apology number if the following equality is satisfied: \[ H_1\colon n = \sum_{i=0}^{m-1} a_i b^i = \sum _{i=0}^{m-1} F(a_i). \qquad \text{($F$-Hardy's apology equation)}\]
		
		Moreover, if $k \in \mathbb N$, we define a natural number $n = \overline{a_{km-1} \cdots a_1 a_0}$ to be a base $b$ $k$-grouped $F$-Hardy's apology number if the following equality holds: $$ H_k \colon n = \sum_{i=0}^{km-1} a_i b^i = \sum _{i=0}^{m-1} F( \overline{a_{ki+k-1} \cdots a_{ki+1}a_{ki}}).$$ $$ \qquad \text{($k$-grouped $F$-Hardy's apology equation)}$$
		
	\end{definition}
	
	The main result of this section is to show that for a given function $F: \mathbb N_0 \rightarrow \mathbb N$, the cardinality of the set of all base $b$ $k$-grouped $F$-Hardy's apology numbers is finite. For the ease of our argument, first, we give the following lemma:
	
	\begin{lemma}
		
		\label{pnh}
		
		For each natural number $m$, the following statements hold:
		
		\begin{enumerate}
			
			\item If $m\geq 7$, then $2^{m-1} - m^2 > 0$ and 7 is the best lower bound for $m$ such that this inequality holds.
			
			\item If $m \geq 4$, then $3^{m-1} - m^2 > 0$ and 4 is the best lower bound for $m$ such that this inequality holds.
			
			\item If $m \geq 3$, then $4^{m-1} - m^2 > 0$ and 3 is the best lower bound for $m$ such that this inequality holds.
			
			\item If $m \geq 2$ and $b \geq 5$, then $b^{m-1} - m^2 > 0$ and 2 is the best lower bound for $m$ such that this inequality holds.
			
		\end{enumerate}
		
		\begin{proof}
			
			(1): The proof is by induction on $m\geq 7$. Set $a_m = 2^{m-1} - m^2$. Clearly, $a_7 = 2^6 - 7^2 = 15 > 0$. Now imagine $a_m > 0$ and we prove that $a_{m+1} > 0$. Since $k \geq 7$, we have that \[a_{m+1} - a_m = 2^m - (m+1)^2 - 2^{k-1} + m^2 = 2^{m-1} -2m -1 > 2^{m-1} - m^2 = a_m > 0.\] This already means that $a_{m+1} > 0$ and the proof by induction is complete. Note that $a_6 = 2^5 - 6^2 = 32 - 36 = -4$.
			
			(2): Set $b_m = 3^{m-1} - m^2$. It is clear that $b_m > a_m$. So for $m\geq 7$, we have that $b_m > 0$. On the other hand, $b_6 = 3^5 - 6^2 = 207$, $b_5 = 3^4 - 5^2 = 56$, $b_4 = 3^3 - 4^2 = 11$, while $b_3 = 3^2 - 3^2 = 0$ and this is what we wanted to show.
			
			(3): Define $c_m = 4^{m-1} - m^2$. Clearly, $c_m > b_m$. So by (2), $c_m > 0$ holds for each $m \geq 4$. But $c_3 = 4^2 - 3^2 = 7$, while $c_2 = 4 - 2^2 = 0$.
			
			(4): Set $d_m = 5^{m-1} - m^2$. Obviously by (3), $d_m > 0$ holds for each $m \geq 3$. Also, note that $d_2 = 5 - 2^2 = 1$, while $d_1 = 0$. Now, define $f_m = b^{m-1} - m^2$, where $b \geq 6$. Clearly, $f_m > d_m$ for any $m \geq 2$ and finally, $f_1 = 0$ and the proof is complete.
		\end{proof}
		
	\end{lemma}
	
	\begin{theorem}
		
		\label{FsumPA}
		
		Let $F: \mathbb N_0 \rightarrow \mathbb N$ be a function, $b \in \mathbb N -\{1\}$ and $a_i$ be a non-negative integer such that $a_i \leq b-1$, for all $0 \leq i \leq km-1$, where $k$ is a fixed positive integer and $m$ is an arbitrary positive integer. Set $$s_k = \max \{F(0), F(1), F(2), \ldots, F(\underbrace{\overline{(b-1) \cdots (b-1)(b-1)}}_{k-\textrm{times}})\}.$$ Then the following statements hold:
		
		\begin{enumerate}
			
			\item[(H)] \label{Equationk} If $a_{km-1} \neq 0$, then the following equation has finitely many solutions: $$ \sum_{i=0}^{km-1} a_i b^i = \sum _{i=0}^{m-1} F( \overline{a_{ki+k-1} \cdots a_{ki+1}a_{ki}}).$$
			
		\end{enumerate}
		
		Moreover, if an algorithm is designed to find all the solutions of the equation mentioned in the statement (H), then the solutions are needed to be checked for all $ n = \sum_{i=0}^{km-1} a_i b^i\leq b^{km-1}$, where $m$ is as follows:
		
		\begin{enumerate}
			
			\item If $k=1$ and $b=2$, then $m = \max\{7, s_1\}$.
			
			\item If $k=1$ and $b=3$, then $m = \max\{4,s_1\}$.
			
			\item If $k=1$ and $b=4$, then $m = \max\{3,s_1\}$.
			
			\item If $k$ is arbitrary and $b \geq 5$, then $m = \max\{2,s_k\}$.
		\end{enumerate}
		
	\end{theorem}
	
	\begin{proof}
		
		Let $s_k = \max \{F(0), F(1), F(2), \ldots, F(\underbrace{\overline{(b-1) \cdots (b-1)(b-1)}}_{k-\textrm{times}})\}$. First, we prove that the statement (H) holds. It is clear that $$ \sum _{i=0}^{m-1} F( \overline{a_{ki+k-1} \cdots a_{ki+1}a_{ki}}) \leq ms_k,$$ while $ \sum_{i=0}^{km-1} a_i b^i \geq b^{km-1}$, since $a_{km-1} \neq 0$. Note that $\lim b^{km-1} / ms = + \infty$, when $m \longrightarrow +\infty$. Consequently, there is a natural number $M$ such that for $m \geq M$, we have  that $b^{km-1} > ms_k$. This means that for such an $M$, if $m \geq M$, then we have $$ \sum_{i=0}^{km-1} a_i b^i > \sum _{i=0}^{m-1} F( \overline{a_{ki+k-1} \cdots a_{ki+1}a_{ki}}).$$ Therefore, the equation mentioned in the statement (H) has finitely many solutions. Now we go further to prove the other statements of the theorem.
		
		(1): Let $m \geq \max\{7,s\}$. By Lemma \ref{pnh}, $2^{m-1} - m^2 > 0$. Also it is clear that $2^{m-1} - ms \geq 2^{m-1} - m^2$. So for this case, according to the proof of the statement (H), if $n$ is a solution for Hardy's Apology Functional, then $n  \leq 2^{m-1}$.
		
		(2) \& (3): By considering Lemma \ref{pnh}, the proof of the statements (2) and (3) is similar to the proof of the statement (1), and therefore it is omitted.
		
		(4): Let $m \geq \max\{2,s\}$ and $b \geq 5$. Then by Lemma \ref{pnh}, $b^{m-1} - m^2 > 0$. But $b^{km-1} - ms \geq b^{m-1} - m^2$. Therefore, for this general case, if $n$ is a solution for a $k$-Hardy's Apology Functional, then $n \leq b^{km-1}$.
	\end{proof}

	\begin{example}
		Our definition for a base $b$  $k$-grouped $F$-Hardy's apology number is inspired by the following historical examples.
		
		\begin{enumerate}
			
			\item A number is called M\"{u}nchhausen in base $b$, if $n =  \overline{a_{m-1} \cdots a_1 a_0} = \sum_{i=0}^{m-1} a_i b^i $, then $n = \sum _{i=0}^{m-1} a_{i}^{a_i}$. Perhaps the most famous  M\"{u}nchhausen number in base 10 is $3435 = 3^3 + 4^4 + 3^3 + 5^5 $. For more on M\"{u}nchhausen numbers, one can refer to the paper \cite{vanBerkel2009} by Daan van Berkel.
			
			\item A number is called factorial in base $b$, if $n =  \overline{a_{m-1} \cdots a_1 a_0} = \sum_{i=0}^{m-1} a_i b^i $, then $n = \sum _{i=0}^{m-1} a_{i} !$. For example, one can easily check that $40585 = 4! + 0! + 5! + 8! +5!$. Poole \cite{Poole1971} proved in an `exhaustive' way that the only factorial numbers in base $10$ are ${1, 2, 145, 40585}$. He asked for a better proof method, which is provided in this paper.
			
			\item A number is said to be subfactorial in base $b$, if $n =  \overline{a_{m-1} \cdots a_1 a_0} = \sum_{i=0}^{m-1} a_i b^i $, then $n = \sum _{i=0}^{m-1} !a_{i} $. Note that the subfactorial of a natural number $n$, denoted by $!n$, is defined as follows: $$!n = n! \sum_{i=0}^n (-1)^i / i!.$$ An important example for subfactorial numbers is the following: $$148349 = !1 + !4 + !8 + !3 + !4 + !9.$$
			
			\item A number is $d$-perfect 1-digit invariant in base $b$, if $n =  \overline{a_{m-1} \cdots a_1 a_0} = \sum_{i=0}^{m-1} a_i b^i $, then $n = \sum _{i=0}^{m-1} a_{i}^d$. For example, Hardy on page 25 of his famous historical book on mathematics mentions that ``there are just four numbers (after 1) which are the sums of the cubes of their digits \cite{Hardy1940}''. He only mentions cubic perfect (3-perfect) 1-digit invariants 153, 370, 371, 407. Quadric perfect 1-digit invariants also exist. For example, for $d=4$, we have $1634=1^4+6^4,+3^4+4^4$ other numbers belonging to this class are 8208 and 9474. For $d=5$, we have 4150, 4151, 54748, 92727, 93084, and 194979. If $d=6$, the only result below $n<10^6$ is 548834. These numbers are also coined as narcissistic, or Armstrong numbers \cite{Looijen2019, Weisstein2005}. 
			
			\item A number is a dual of a $d$-perfect 1-digit summative number in base $b$, if $n =  \overline{a_{m-1} \cdots a_1 a_0} = \sum_{i=0}^{m-1} a_i b^i $, then $n = \sum _{i=0}^{m-1} d^a_{i}$. For example, all such numbers $n<10^6$ are for $d=3,  12=3^1+3^2$. For $d=4,  4624=4^4+4^6+4^2+4^4, 595968=4^5+4^9+4^5+4^9+4^6+4^8$.  For $d>4$ there are no dual $d$-perfect 1-digit numbers. Note the difference among use of the variable `$d$': in previous and next examples the $d$ is for exponents, while in the current example it is used as the base number.
			
			\item A $(k,d)$-digit summative number in base $b$, is $n =  \overline{a_{m-1} \cdots a_1 a_0} = \sum_{i=0}^{m-1} a_i b^i $, with $n = \sum_{i=0}^{km-1} \overline{a_{ki+k-1} \cdots a_{ki+1}a_{ki}}^{d}$.  Examples for $b = 10$ with $k= 2$ and $d= 3$ are $165033 = 16^3 +50^3 + 33^3$, $221859 = 22^3 + 18^3+59^3$, $341067=34^3+10^3+67^3$, $444664=44^3+46^3+64^3$, and $487215=48^3+72^3+15^3$.
			
		\end{enumerate}
	\end{example}

	\begin{remark}
		Theorem \ref{FsumPA} is strong because it does not put any condition on $F$, but it is quite important to note that it needs $k$ to be a fixed positive integer. The reason for this is that if we suppose $k$ to be an arbitrary positive integer, then for a specific $F$, all $k$-grouped $F$-Hardy's apology numbers obtained from the $H_k$ equation in Theorem \ref{FsumPA} may have infinitely many solutions. We show this in the next section.
	\end{remark}
	
	\section{Armstrong Numbers and their Generalization}\label{sec:Armstrong}
	
	Armstrong numbers have a pretty interesting history. As Lionel Deimel says in \cite{Deimel2010}, it wasn't clear who exactly this mysterious Armstrong behind Armstrong numbers is. Apparently, someone sent an email to Deimel claiming he is the Armstrong. In the email, he says
	\begin{displayquote}
		``In the mid 1960s -- probably around 1966 -- I was teaching an elementary course in Fortran and computing in general at The University of Rochester, and “invented” Armstrong Numbers as an exercise for my students. I still have the original coffee-stained paper that was the master copy for the homework assignment..."
	\end{displayquote}
	and also sent a copy of his paper \cite{Armstrong1966} to Deimel. In it Armstrong defines four types of Armstrong numbers, from which the generalization of the last one currently is known as \textit{Armstrong numbers}:
	
	\begin{definition}
		
		\label{Armstrongnumbers}
		
		Suppose that we have an $n-$digit number in base $b$. This number is called a base-$b$ Armstrong number of $n^{\text{th}}$ order if it is equal to the sum of the $n^{\text{th}}$ power of its digits in base $b$.
	\end{definition}
	
	We can imply from \ref{FsumPA} that there are finitely many Armstrong numbers in any base $b$. This can be verified for small values of $b$: in \cite{Miller1992}, Miller and Whalen show that 12, 22, and 122 are the only Armstrong numbers in base 3 and 130, 131, 203, 223, 313, 332, 1103, and 3303 are the only Armstrong numbers in base 4. 
	
	\begin{example}\label{base10armstrong}
		There are three base-$10$ Armstrong numbers of $4^{\text{th}}$ order \cite{Weisstein2005}:
		\begin{align*}
		1634 &= 1^4+ 6^4+ 3^4 + 4^4\\
		8208 &= 8^4+ 2^4 + 0^4 + 8^4\\
		9474 &= 9^4 + 4^4 + 7^4 + 4^4.
		\end{align*}
	\end{example}
	
	Weisstein \cite{Weisstein2005} states that D. H. Winter computed \cite{WinterTable} the existing $88$ Armstrong numbers in base  $b = 2, \dots, 16$ and they are all of order $n \leq 60$. See also \cite{AntalanBadua2014}, or the OEIS database. The $11 - 16$ numbers are in sequences A161948 - A161953.
	
	The following proposition is taken from \cite{Wong2013} and may serve as an example for encryption \cite{Bansode2014}  application of Armstrong numbers, e.g. with two keys $x$ and $y$ if brought together computation should match according to formula (2.1), as executed in a chip of a security lock.
	
	\begin{proposition}
		
		\label{ewong}
		
		There are infinitely many positive integers $x=\overline{x_{k-1}\dots x_1x_0}$ and $y=\overline{y_{k-1}\dots y_1y_0}$ (in base $10$) such that:
		\begin{align}
		\overline{xy} = x^2 + y^2, \label{eq:wong1}
		\end{align}
		where $\overline{xy}$ is the concatenation of $x$ and $y$. Note the invariance of \eqref{eq:wong1} in base $b = 10$ for  $x \rightarrow 10^k - x$, where $k$ is the block length sampled in the numbers.
		
		\begin{proof}
			
			A solution to the equation \eqref{eq:wong1} is
			\begin{align}
			a=4,\quad b=10^{4u},\quad x = \displaystyle \frac{a}{17}\cdot(ab-1),\quad y = \displaystyle \frac{a}{17}\cdot (a+b)\label{eq:wong-soln},
			\end{align}
			where $u=4t+3$, for $t \geq 0$. This shows that we have found infinitely many solutions to the equation \eqref{eq:wong1}. \end{proof} 
	\end{proposition}
	
	\begin{example} (Some examples for \eqref{eq:wong1} in Proposition \ref{ewong}) Easy examples for the equation \eqref{eq:wong1} are:
		\begin{align}
		12^2+33^2 &= 12 \; 33\nonumber\\
		88^2+33^2 &= 88 \; 33.\nonumber
		\end{align}
		
		Larger examples for \eqref{eq:wong1} are:
		
		\begin{align}
		9412^2 + 2352^2 &= 9412 \; 2352\nonumber\\
		941176470588^2+235294117648^2 &=941176470588 \; 235294117648, \label{eq:wong2}
		\end{align}

		and finally,
		
		\begin{align}
		x = 9411764705882352941176470588 \text{ and } y = 2352941176470588235294117648. \label{eq:wong3}
		\end{align}
		
		In fact, equations \eqref{eq:wong2} and \eqref{eq:wong3} are generated by plugging $m=0$ and $m=7$ in \eqref{eq:wong-soln}. \end{example}
	
	\begin{remark}
		
		The Fermat prime $17$ appears in \eqref{eq:wong-soln}, and that happens to be a connection between the solutions to \eqref{eq:wong1} and Fermat primes. In \cite{Tito2013}, Tito Piezas III conjectured that a similar solution to \eqref{eq:wong-soln} can be found for other Fermat primes such as $257$ and $65537$. In other words, he conjectured that 
		\begin{align}
		a=16,\quad b=10^{64u}, \quad x &= \tfrac{a}{257}(ab-1), \quad y = \tfrac{a}{257}(a+b)\label{eq:wong-soln2}\\
		a=256,\quad b=10^{16384u},\quad x &= \tfrac{a}{65537}(ab-1),\quad y = \tfrac{a}{65537}(a+b)\label{eq:wong-soln3}
		\end{align}
		are solutions to \eqref{eq:wong1} for all $u=4t+3$, $t\geq 0$. V. Ponomarenko \cite{Tito2013} proved this conjecture.  In the Computational Appendix below is shown that Piezo's iteration implicitly uses two consecutive Fermat numbers and can be used to find unwieldy large Narcissistic numbers. For instance, for the Fermat prime 65537 the numbers $x$ and $y$ each have 100,000 digits and produce a correct result.
	\end{remark}

	\begin{example}
		In the previous example, using the notation of Theorem \eqref{FsumPA}, we saw that if $m=2$ is fixed and $k$ is arbitrary, then there are infinitely many solutions to
		\begin{align*}
		\sum_{i=0}^{km-1} a_i b^i = \sum _{i=0}^{m-1} F( \overline{a_{ki+k-1} \cdots a_{ki+1}a_{ki}}),
		\end{align*}
		where $F(x)=x^2$. The same is true for $m=3$ and $F(x)=x^3$. In fact, we can construct infinitely many solutions for
		\begin{align}
		\overline{xyz} = x^3 + y^3 + z^3,
		\end{align}
		given one initial solution. For instance, in \cite{Vitalis2012}, we see that starting from the initial solution
		\begin{align*}
		153   &=   1^3 + 5^3 + 3^3,
		\end{align*}
		we can construct the general solution
		\begin{align*}
		\overline{1\underbrace{66\cdots 6}_{l\text{ times}}}^3 + \overline{5\underbrace{00\cdots 0}_{l\text{ times}}}^3 + \overline{3\underbrace{33\cdots 3}_{l\text{ times}}}^3 &= \overline{166\cdots 6500\cdots 0333\cdots 3}.
		\end{align*}
		There are many more examples of this kind in \cite{Vitalis2012}.
	\end{example}

	\section{A Generalization of Wells Numbers}\label{sec:Wells}
	
	David Wells, in his book \cite{Wells1987} on page 98, explains that (after $n=1$) for $n = 22$, 23 and 24 only, the number of digits in $n!$ is equal to $n$. This example motivates us to give the following general definition. Let, in this paper, the number of digits of $n$ in base $b$ be denoted by $D_b(n)$. It is clear that if $n = \overline{a_{m-1} \cdots a_1 a_0} = \sum_{i=0}^{m-1} a_i b^i$ is the representation of the natural number $n$ in base $b$, then $D_b(n)=m$, by definition. It is easy to see that $D_b(n) = \lfloor 1+\log_b(n) \rfloor$.
	
	\begin{definition}
		
		\label{Wellsnumbers}
		
		Let $F: \mathbb N \rightarrow \mathbb N$ be a function and $F(n) = \sum_{i=0}^{m-1} a_i b^i$ be the representation of $F(n)$ in base $b\geq 2$. We call a number to be a base $b$ $F$-Wells number, if $n = D_b(F(n))(=m)$, i.e., $n$ is the fixed point element of the function $D_b \circ F$.
		
	\end{definition}
	
	\begin{proposition}
		
		\label{Henk}
		
		Let $F: \mathbb N \rightarrow \mathbb N$ be a function. A natural number $n$ is a base $b$ $F$-Wells number if and only if $b^{n-1} \leq F(n) < b^n$.
		
		\begin{proof}
			A number $n$ is a base $b$ $F$-Wells number if and only if $n=\lfloor 1+\log_b(F(n)) \rfloor$, which is equivalent to say that  $n \leq 1 + \log_b(F(n)) < n+1$ and the proof is complete.
		\end{proof}
		
	\end{proposition}
	
	\begin{corollary}
		
		\label{FWellsnumbers}
		
		Let $F: \mathbb N \rightarrow \mathbb N$ be a function. There are finitely many bases $b$ $F$-Wells numbers if one of the following statements holds:
		
		\begin{enumerate}
			
			\item There is a natural number $N$ such that $n \geq N$ implies that $F(n) \geq b^n$.
			
			\item There is a natural number $N$ such that $n \geq N$ implies that $F(n) < b^{n-1}$.
			
		\end{enumerate}
		
		\begin{proof}
			If one of the above conditions holds, then we have at most $N-1$ base $b$ $F$-Wells numbers.
		\end{proof}

	\end{corollary}
	
	\begin{example}
		\begin{enumerate}
			
			\item If $F\in \mathbb Q(X)$ is a positive integer-valued rational function, i.e. $F(n)\in \mathbb N$ for all $n\in \mathbb N$, then it is clear that the number of $F$-Wells numbers is finite. For example, if $F(n) = n^4$, then it is clear that $10^{n-1} \leq n^4 < 10^n$ if and only if $n=1,2$.
			
			\item Let $F(n)=n!$. Then there are finitely many decimal $F$-Wells numbers and the proof is as follows: Since $n! \geq e(n/e)^n$ \cite[Exercise 10.14 p. 399]{Apostol1967} and $28/e > 10$, it is clear that if $n \geq 28$, then $n! \geq 10^n$. Therefore, if $n$ is a decimal $n!$-Wells number, then $n<28$. In fact, a simple computation shows the only $n!$-Wells numbers are 1, 22, 23, and 24.
			
			\item The only decimal $n^n$-Wells numbers are 1, 8, and 9. Because if $n\geq 10$, then $n^n \geq 10^n$. So, if $n$ is a decimal $n^n$-Wells number, then $n < 10$, to be checked as follows:
			
			$1^1 = 1$, $2^2 = 4$, $3^3 = 27$, $4^4 = 256$, $5^5 = 3125$, $6^6= 46656$, $7^7 = 523543$, $8^8 = 16777216$, and $9^9 = 387420489$.
			
			\item The only decimal $!n$-Wells numbers in base 10 are $n=24, 25$. Note that $D_{10}(!23) = 22$, while $D_{10}(!26) = 27$.
			
		\end{enumerate}
	\end{example}

	\section{A Generalization of Dudeney Numbers}\label{sec:Dudeney}
	
	Henry Ernest Dudeney (1857--1930) in his book \cite{Dudeney1967} on page 36, introduces some special numbers with the property that the cube root of these numbers is equal to the sum of their digits, which is equivalent to say that these numbers are equal to the sum of the digits of their cube. Today these numbers are called Dudeney numbers. Example of Dudeney numbers include 512 and 19683, since $\sqrt[3]{512} = 5+1+2 = 8$ and $\sqrt[3]{19683} = 1+9+6+8+3
	= 27$. For a better notation, we recall that if a non-negative integer $n$ has a representation $n = \sum_{i=0}^{m-1} a_i b^i$ in base $b \geq 2$, then the sum of the digits of $n$ is denoted by $S_b(n) = \sum_{i=0}^{m-1} a_i$ \cite[Theorem 6.5.1]{AndreescuAndrica2009}.
	
	\begin{definition}
		
		\label{Dudeneynumbers}
		
		Let $F: \mathbb N \rightarrow \mathbb N$ be a function and $F(n) = \sum_{i=0}^{m-1} a_i b^i$ be the representation of $F(n)$ in base $b\geq 2$. We call a number $n$ to be a base $b$ $F$-Dudeney number, if $n = S_b(F(n))$, where $S_b(F(n)) = \sum_{i=0}^{m-1} a_i$.
		
	\end{definition}
	
	\begin{theorem}
		
		\label{Dudeneythm}
		Let $F: \mathbb N \rightarrow \mathbb N$ be a function and $n$ a base $b$ $F$-Dudeney number. Then the following statements hold:
		
		\begin{enumerate}
			
			\item $F(n) \geq b^{\frac{n-b+1}{b-1}}$
			
			\item If there is an $N\in \mathbb N$ such that $n\geq N$ implies that $F(n) < b^{\frac{n-b+1}{b-1}}$, then there are finitely many base $b$ $F$-Dudeney numbers.
			
		\end{enumerate}
		
		\begin{proof}
			
			(1): Since $a_i \leq b-1$ and $S_b(F(n)) = \sum_{i=0}^{m-1} a_i$, we have that $S_b(F(n)) \leq (b-1)m$. But $m$ is the number of the digits of $F(n)$. So $m =  \lfloor 1+\log_b(F(n)) \rfloor$. Since $n$ is a base $b$ $F$-Dudeney number, $n \leq (b-1)\lfloor 1+\log_b(F(n)) \rfloor$. So, we have $n \leq b-1 + (b-1)\log_b(F(n))$ and finally, $F(n) \geq b^{\frac{n-b+1}{b-1}}$.
			
			(2) is just a result of (1).
		\end{proof}
		
	\end{theorem}
	
	\begin{corollary}
		
		\label{DudeneyCor}
		Let $F: \mathbb N \rightarrow \mathbb N$ be a function such that $\displaystyle \lim_{n \rightarrow +\infty} \frac{F(n)}{b^{\frac{n-b+1}{b-1}}} = 0$. Then there are finitely many base $b$ $F$-Dudeney numbers. In particular, if $F\in \mathbb Q[X]$ is a polynomial function such that $F(n)\in \mathbb N$ for all $n\in \mathbb N$, then there are finitely many base $b$ $F$-Dudeney numbers.
	\end{corollary}
	
	\begin{example}
		
		\begin{enumerate}
			
			\item A Dudeney number is a positive integer such that the sum of its decimal digits is equal to the cube root of the number. There are exactly seven such integers (sequence A061209 in the OEIS): 0, 1, 512, 4913, 5832, 17576, and finally, 19683.
			
			\item Let $\mathcal{F}: \mathbb N \rightarrow \mathbb N$ be the Fibonacci sequence, the sequence that is defined as follows: $\mathcal{F}_{n+2} = \mathcal{F}_{n+1} + \mathcal{F}_n$ for all $n\in \mathbb N$ and $\mathcal{F}_1 = 1$,  $\mathcal{F}_2 = 1$ and $\mathcal{F}_3= 2, \mathcal{F}_4 = 3, \mathcal{F}_5 = 5, \mathcal{F}_6 = 8, \mathcal{F}_7 = 13, \mathcal{F}_8= 21, \mathcal{F}_9= 34, \mathcal{F}_{10} = 55.$ The specific Fibonacci-Dudeney numbers are: $\mathcal{F}_1$, $\mathcal{F}_5$, and $\mathcal{F}_{10}$.
			
		\end{enumerate}
		
	\end{example}
	
	Let $k(n)$ be the function that gives the sum of the $F$ of each digit of a positive integer $n$. Also let $m$ be the number of digits in integer $n$. Then $10^{m-1} \leq n < 10^m$. Since 9 is the largest possible digit, we conclude that $k(n) \leq m\cdot s$. Therefore, we have that $$10^{m-1} \leq n \leq ms.$$ In other words, $m$ needs to satisfy this inequality: $$\displaystyle 10 \leq (ms)^{\frac{1}{m-1}}$$
	
	\section{The Powers of the Sum of the Digits of a Number}\label{sec:PowersSumDigits}
	
	In this short section, we show that the number of natural numbers which are equal to the sum of their digits raised to a specific power is finite.
	
	\begin{theorem}
		
		\label{PowersSumDigits}
		
		Let $n$ be a natural number and set $S_b(n)$ to be the sum of the digits of the number $n$ in base $b$. Let $p \geq 2$ be a natural number and define $\phi_p : \mathbb N \rightarrow \mathbb N$ with $\phi_p (n) = (S_b(n))^p$. Then the number of natural numbers satisfying the equality $\phi_p (n) = n$ is finite. Moreover, if $\phi_p (n) = n$ then $n \leq b^{p^2}$.
		
		\begin{proof}
			Let $n = \sum_{i=0}^{m-1} a_i b^i$ be the representation of the number $n$ is base $b$. It is clear that $S_b(n) \leq (b-1)m$, where $m = \lfloor 1+\log_b(n)\rfloor$. So if a number $n$ satisfies the equality $\phi_p (n) = n$, then it needs to satisfy the inequality $n \leq (b-1)^p (\lfloor 1+\log_b(n) \rfloor)^p$.
			
			Now define a real function $f(x) = x - (b-1)^p (1+\log_b(x))^p$. It is clear that $\displaystyle f^{\prime}(x) = 1- \frac{p(b-1)^p}{x\ln{a}} (1+\log_b(x))^{p-1}$, which shows that for large enough real numbers $x$, the function $f$ is increasing.
		\end{proof}
		
	\end{theorem}
	
	\section{Computational Appendix}
	
	\begin{appendix}
		To check results as obtained in this paper, we show a few algorithms to do the work. The programming language used is the Maple programming language. In the first paper on finding narcissistic numbers, the BASIC language was used \cite{Ecker1984}. A conversion $C$ of a decimal integer to a list of its single digits is: \[C := (b, n) \rightarrow convert(n, base, b)\] and its inverse is Maple's concatenation of digits: \[J := n \rightarrow Joinsequence(n)\]
		The sum of a list $a$ of numbers is: \[S := a \rightarrow add(i, i \; \epsilon \; a)\]
		The operator $Db$ for the digit length of a number $n$ is: \[Db := n \rightarrow length( n )\]
		
		By following Curry and Feys' \cite{CurryFeys1974} notation of operators and operator sequences, we deviate from the common use of Category Theory \cite{BirdMoor1997} to express algorithms. 
		This considerably simplifies notation, as follows:\\
		
		\begin{description}
			\item[H] Hardy's Apology numbers satisfy the fixed points:  $n = S(F(C(n)))$.
			\item[D] Dudeney numbers are the fixed points:  $n = F(S(C(n)))$. 
			\item[W] Wells numbers satisfy: $n = Db(F(n))$.
		\end{description}
		
		Permutation of the order of operators might be beneficial for the discovery of new numbers. For instance, in the reverse order of operators of the Wells numbers $n = F(Db(n))$ with 
		$f \epsilon F,  f := x \rightarrow x^4$, all fixed points are $n = 0, 1, 32, 243, 1024$.\\
		
		The algorithm to find the fixed points in Hardy's Apology Theorem is  $n = S(F(C(n)))$ needs precautions 
		for data handling from one operator to the next, and with the fixed point criterion at the end as a haltings criterion, as follows:\\
		
		\begin{verbatim}
		HApolTh := proc(b, k, n, F)  local a, x; 
		a := C(b^k, n);   x := [ seq( F(a[i]), i = 1 ... length(a)) ]; 
		if n = S(x) then    n   end if 
		end proc
		\end{verbatim}
		
		For example, $f \epsilon F, f := x \rightarrow x^5$, seq(HApolTh(10, 1, n, f), n = 1 ... $10^5$) exhausts all the fixed points $n =  1, 4150, 4151, 54748, 92727, 93084$, as seen from our upper bounds limit in the Apology Theorem.
		
		The search can be limited as was urgently needed with historical computers \cite{Lamb1986} by use of our Hardy's Apology Theorem \ref{FsumPA}. Take  $f \epsilon F, f := x \rightarrow x^2$, in base $b = 10$ is $9$ the maximum digit. So, the inequality $m (9^2) < 10^m$, where $m$ is the length of the number, is false if $m=1,2$. Hence, single and two digits numbers cannot be narcissistic, but with $m=3$, we have $243<1000$. Concluding, the 3-digit numbers in base 10 cannot have squared  1-digit sums larger than 243, which is the limit for search. Continuing this we even could design automated upper bounds for search with digit blocks $k=2,3, ...$, etc. 
		
		Piezas' computation of narcissistic numbers as formulated by \cite{Tito2013} is speeded-up by his inclusion of Fermat steps, with the Ith Fermat number of $n$ and its predecessor:\\
		
		\begin{verbatim}
		Piezas := proc(k, n)  local a, b, fe, l, m, x, y; 
		
		fe := IthFermat(n);        l := (fe-1)/4; m := 4*k+3;  
		a := IthFermat(n-1)-1; b := 10^(l*m); 
		x := a*(a*b-1)/fe;           y := a*(a+b)/fe; 
		print(k, m, n, a, l, fe, x, y, x^2+y^2)
		end proc
		\end{verbatim}
		
		The algorithm is appropriate for fast computing of extremely long narcissistic numbers with hundreds of thousands of digits (on a PC).
		
		As quoted in our introduction  Hardy's remark ``8712 and 9801 are the only four-figure numbers which are integral multiples of their reversals'', he introduced the eigenvalue $\lambda$ for the consecutive application of operations $C$, permutation $P$ and join $J$  as follows:
		
		$$2178 = J(P(C(8712)  \rightarrow 8712 = \lambda J(P(C(8712) \rightarrow \lambda = 4$$
		
		Using eigenvalues is studied to its full extent by Lara Pudwell \cite{Pudwell2007} and Sutcliffe \cite{Sutcliffe1966}. Lara Pudwell in her paper \cite{Pudwell2007} relates the story about Hardy's short-sighted second remark above. Sutcliffe \cite{Sutcliffe1966} generalized the problem to reversals in any base. 
		
	\end{appendix}
	
	\subsection*{Acknowledgments}
	
	The second named author is supported by the Department of Engineering Science at the Golpayegan University of Technology and his special thanks go to the Department for providing all necessary facilities available to him for successfully conducting this research.
	
	\bibliographystyle{plain}

\end{document}